\documentclass[12pt]{amsart}

\usepackage{graphicx}
\usepackage{amsthm}
\swapnumbers

\newtheorem{Thm}{Theorem}
\newtheorem{MThm}[Thm]{Main Theorem}
\newtheorem{Def}[Thm]{Definition}

\newtheorem{Lem}[Thm]{Lemma}

\newtheorem{Rem}[Thm]{Remark}
\newtheorem{Claim}[Thm]{Claim}

\begin{document}

\title{Meridional almost normal surfaces in knot complements}
\author{Robin T. Wilson}

\address{Department of Mathematics and Statistics\\
California State Polytechnic University\\
Pomona, CA 91768}
\email{robinwilson@csupomona.edu}


\thanks{Research supported by UC President's Postdoctoral Fellowship Program and Department of Mathematics at the University of California, Santa Barbara}

\begin{abstract}
Suppose $K$ is a knot in a closed 3-manifold $M$ such that $\overline{M-N(K)}$ is irreducible.  We show that  for any integer $b$ there exists a triangulation of $\overline{M-N(K)}$ such that any weakly incompressible bridge surface for $K$ of $b$ bridges or fewer is isotopic to an almost normal bridge surface. 
\end{abstract}

\maketitle

\section{Introduction}
\label{intro}

It was shown independently by M. Stocking \cite{St} and J. H. Rubinstein \cite{Ru} that any strongly irreducible Heegaard splitting for an irreducible 3-manifold is isotopic to an almost normal surface.  Also see \cite{K} for another proof of this result. The concept of a bridge surface for a knot complement is analogous to the idea of a Heegaard surface for a 3-manifold in that the bridge surface is a splitting surface that separates the knot complement into two equivalent and fairly elementary submanifolds.  In addition, the fact that a bridge surface lifts to a Heegaard surface in the 2-fold branched cover of a knot complement gives another important connection between bridge surfaces for knot complements and Heegaard surfaces for 3-manifolds. 

 In the study of bridge surfaces for knots and links the idea of a weakly incompressible splitting surface is immediately analogous to the idea of a strongly irreducible Heegaard surface for a 3-manifold.  In this paper we prove an analog to the main theorem of \cite{St} and Theorem 3 in \cite{Ru}.  We show that any weakly incompressible bridge surface in a 3-manifold is isotopic to an almost normal bridge surface.  
 
 \begin{MThm}
\label{almostnormal}
Let $K$ be a knot in a closed, orientable, irreducible 3-manifold $M$.  Let $N(K)$ be a regular neighborhood of $K$ and suppose that the closure of the complement $M_K=\overline{M-N(K)}$ is irreducible.  Then for any integer $b$ there is a triangulation $\mathcal{T}$ of $M_K$ such that if $S$ is a bridge surface for $K$ of $b$ bridges or fewer that gives an irreducible Heegaard splitting of $M$ and $S_K=\overline{S-N(K)}$ is weakly incompressible, then $S_K$ is properly isotopic in $M_K$ to an almost normal surface with respect to $\mathcal{T}$.
\end{MThm}

The proof will be similar in spirit to that of \cite{St} but the proof here fills in a few missing cases and simplifies the argument by making greater use of edge slides.  A closely related result was proved by David Bachman in \cite{B}, however it requires additional hypotheses.  In Section 2 we briefly introduce some definitions and notation.  The proof of the main theorem is contained in Section 3.  
 
 This research was done while under the support of the UC President's Postdoctoral Fellowship Program and the Department of Mathematics at UC Santa Barbara.  I would like to thank Martin Scharlemann for all of the helpful conversations and many valuable comments, as well as Scott Taylor for his insightful comments about the proof of Lemma \ref{Lemma4}.

\section{Preliminaries}
\label{definitions}

\textbf{Notation}:  If $K$ is a properly embedded 1-manifold in a 3-manifold $M$ then let $M_K=\overline{M-N(K)}$.  If $X$ is any surface in $M$ transverse to $K$ such that $K\cap X\neq \emptyset$, then let $X_K=M_K\cap X$.  For $\mathcal{T}$ a triangulation of a 3-manifold $M$, let $\mathcal{T}_{\partial M}$ denote the restriction of $\mathcal{T}$ to $\partial M$.  \\

The following definition is from \cite{To} and is based on the definition of a $K$-compression body given in \cite{Ba}. 

\begin{Def}[\cite{To}]\textup{
A properly embedded arc $K$ in a 3-manifold $M$ is \textit{boundary parallel} if there is a disk $D$ in the 3-manifold so that $\partial D$ is the end point union of $K$ and an arc in $\partial M$.  The disk $D$ is called a \textit{cancelling disk} for $K$.  A \textit{$K$-handlebody} $(A,K)$ is a handlebody $A$ containing a finite collection of boundary parallel arcs $K$.  When there is little risk of confusion we will also refer to $A_K=\overline{A-N(K)}$ as a $K$-handlebody.  For our purposes, a \textit{$K$-compression body} (W,K) is a compression body $W$ containing a finite collection of arcs $K$ properly embedded in $W$ such that each arc has one end on each of $\partial_{+}W$ and $\partial_{-}W$ and each arc is vertical in the product region $\partial_{-}W \times I \subset W$. 
}\end{Def}  

\begin{Rem}\textup{
Two sets $K$ and $K'$ of boundary parallel arcs in a handlebody $A$ or vertical arcs in a compression body are properly isotopic in $A$ if they have the same cardinality, i.e., $|K|=|K'|$.
}\end{Rem}

\begin{Def}\textup{
A \textit{spine} of a handlebody $A$ is a graph $\Sigma_A$ properly embedded in $A$ such that $A-\Sigma_A$ is a product $\partial A \times I$.  Let $A_K$ be a $K$-handlebody and suppose that $\Sigma_A$ is a spine for handlebody $A$ and $K$ is a collection of boundary parallel arcs.  Let $\alpha$ be a collection of $|K|$ arcs, each connecting $\Sigma_A$ to a single arc of $K$.  Then a regular neighborhood $A'=N(\Sigma \cup \alpha)$ is again a handlebody, and $K$ intersects it in a boundary parallel set of arcs $K'\subset A'$.  If the closure of the region $(A-A')_{(K-K')}$ between them is a product $\partial A_K \times I$ then $\Sigma_{(A,K)}=\Sigma_A \cup \alpha$ is called a \textit{spine of the $K$-handlebody} $(A,K)$.  A spine $\Sigma_W$ of a compression body $W$ is the union of $\partial_{-}W$ and a properly embedded graph such that $W-\Sigma_W$ is a product $\partial W \times I$.    If $(W,K)$ is a $K$-compression body then a spine for $W_K$ will mean a spine for $W$ that is disjoint from $K$.  
}\end{Def}

\begin{Rem}\textup{
It is relatively easy to find such a spine for a $K$-handlebody or $K$-compression body $A_K$.  Choose a spine for $\Sigma$ of handlebody (compression body) $A$ and isotope $K$ so that in the product structure $A-N(\Sigma)=\partial A \times I$, each (boundary parallel) arc of $K$ has a single maximum.  Let $\alpha$ ($\beta$) be a collection of vertical arcs in this product structure, connecting each maximum of $K$ to $\Sigma$.  
}\end{Rem}

\begin{Def}[\cite{STo}]\textup{
Let $K$ be a knot in a closed, orientable 3-manifold $M$ and let $S$ be a Heegaard surface for $M$.  That is, $M=W \cup_{S} W'$, where $W$ and $W'$ are handlebodies in $M$.  If in addition, $W_K$ and $W'_K$ are $K$-handlebodies then we call $S$ a \textit{bridge surface} for $M_K$.  (We will often abuse notation and call the punctured surface $S_K$ a bridge surface as well.)  We call the decomposition $M_K=W_K\cup_{S_K}W_{K}'$ a \textit{bridge splitting} of the 3-manifold $M_K$ and we say that $K$ is in \textit{bridge position} with respect to bridge surface $S$.
}\end{Def} 

\begin{Def}[\cite{STo}]\textup{
Let $K$ be a 1-manifold embedded in $M$ and suppose that $F$ is a properly embedded surface in $M$ so that $F$ is transverse to $K$. A simple closed curve on $F_K$ is \textit{essential} if it doesn't bound a disk or a once punctured disk in $F_K$.  An embedded disk $D \subset M_K$ is a \textit{compressing disk} for a surface $F_K$ if $D \cap F_K=\partial D$ and $\partial D$ is an essential curve in $F$.  A surface $F$ in $M$ is a \textit{splitting surface} for $M$ if we can express $M$ as the union of two 3-manifolds along $F$.  If $F$ is a splitting surface for $M$ then we say that the surface $F_K$ is \textit{weakly incompressible} if any pair of compressing disks on opposite sides of the surface intersect.  If $F_K$ compresses on both sides but is not weakly incompressible then it is called \textit{strongly compressible}. 
}\end{Def}

The study of normal surfaces was first developed by Haken \cite{H}.  The concept of an almost normal surface that is used in this paper first appeared in \cite{Ru}.  

\begin{Def}\textup{
Let $S$ be a triangulated surface and let $c$ be a curve on  $S$.  Assume that $c$ is transverse to the 1-skeleton of the triangulation.  A curve $c$ in $S$ is called \textit{normal} if the intersection of $c$ with any triangle of the triangulation contains no closed curves and no arcs with both endpoints on the same edge.
}\end{Def}

\begin{Def}[\cite{H}]\textup{
Let $M$ be a triangulated 3-manifold.  A \textit{normal triangle} in a tetrahedron of the triangulation is an embedded disk that meets three edges and three faces of the tetrahedron.  A  \textit{normal quadrilateral} is an embedded disk in a tetrahedron that meets four edges and four faces of the tetrahedron.  Normal triangles and quadrilaterals are called  \textit{normal disks}.  Normal disks meet the faces of the boundary of a tetrahedron in normal arcs. 
}\end{Def}  

\begin{Def}[\cite{Ru}]\textup{
Let $M$ be a triangulated 3-manifold.  An embedded surface $S\subset M$ is a \textit{normal surface} if it meets each tetrahedron in a disjoint collection of normal disks.  A surface $S$ is \textit{almost normal} if $S$ meets each tetrahedron of the triangulation in a collection of normal disks, but in one tetrahedron there is exactly one exceptional piece.  This exceptional piece is either a normal octagon, or it is an annulus consisting of two normal disks with a tube between them that is parallel to an edge of the 1-skeleton.
}\end{Def}

The proof of the main theorem relies heavily on the idea of thin position, first introduced by Gabai \cite{G}.

 \begin{Def}[\cite{St}]\textup{
 Let $M_K=W_K  \cup_{S_K} W'_K$ denote a bridge splitting of $M_K$.  Given spines $\Sigma_{(W,K)}$ and $\Sigma_{(W',K)}$ for the $K$-handlebodies $(W,K)$ and $(W',K)$ respectively, there is a diffeomorphism $S_K \times (0,1) \simeq M_K-N(\Sigma_{(W,K)} \cup \Sigma_{(W',K)})$.  For $t\in(0,1)$ denote the surface corresponding to $S_K \times \{t\}$ by $S_t \subset M_K$.  A \textit{standard singular foliation} $F$ of $M_K=W_K  \cup_{S_K} W'_K$ extends this structure to all of $M_K$ by adding two singular leaves $S_0$ and $S_1$, called the top and bottom leaves.  All leaves meet the the torus $\partial M_K$ in the standard foliation in meridian circles.  The top and bottom singular leaves consist of the union of the spines of the $K$-handlebodies $W_K$ and $W'_K$ respectively, and the meridian circle of $\partial M_K$ corresponding to each of the $n$ endpoints of $\Sigma_{(W,K)}$ and $\Sigma_{(W',K)}$.  There is a height function $h:M \rightarrow [0,1]$ associated with the standard singular foliation given by the map that sends all points on a leaf $S_t$ together with the incident meridian disks of $N(K)$ to the point $t$ in $[0,1]$.
}\end{Def}

 \begin{Def}[\cite{Th}, \cite{St}]\textup{
 Assume that $T$ is a collection of arcs properly embedded in $M_K$ and is in general position with respect to $F$, a standard singular foliation of $M_K$.  That is, all but a finite number of leaves of $F$ intersect $T$ transversally, every leaf in $F$ has at most one point of tangency with $T$, and $T$ is disjoint from the singular subarcs of the singular leaf.  If a leaf has a point of tangency with $T$ call it a \textit{tangent leaf} and all other leaves \textit{transverse leaves}.  Between each two adjacent tangent leaves choose a transverse leaf $L_i$.  Define the \textit{width of a fixed embedding of $T$ with respect to $F$} to be the sum over $i$ of (the number of times $T$ intersects $L_i$).  If $T$ has been properly isotoped to minimize its with with respect to $F$ then we say that $T$ is in \textit{thin position} with respect to $F$.    
}\end{Def}

\begin{Def}[\cite{Th}, \cite{St}]\textup{
Let $T$ be in thin position with respect to standard singular foliation $F$.  Then as we move down the foliation from the top the arcs will form a sequence of maxima with respect to $F$, then a sequence of minima, and so on.  We will call a leaf in a region where the sequence shifts from maxima to minima a \textit{thick leaf} and we will call such a region a \textit{thick region}.  An \textit{upper (lower) disk} $D$ for a transverse leaf $L$ of $F$ is a disk in $int(M)-T$ such that $\partial D= \alpha \cup \beta$ where $\alpha$ is an arc embedded in $L$, $\beta$ is a subarc of $T$, $\partial \alpha =\partial \beta$, $D-\alpha$ intersects $L$ transversely, and a small neighborhood of $\alpha$ lies above (below) $L$.   
}\end{Def}

For the proof of the main theorem we will need the following Lemmas.  The first theorem is proved by Stocking in \cite{St}.

\begin{Lem}[Lemma 1, \cite{St}]
\label{normal}
Let $S$ be an almost normal surface in an irreducible 3-manifold.  Suppose that $S$ is incompressible to one side.  Then $S$ is isotopic to a normal surface that does not intersect $S$ and that does not contain $S$ to the incompressible side. 
\end{Lem}

A version of the following theorem was proved for strongly irreducible Heegaard surfaces by Casson and Gordon in \cite{CG} and has been adapted to the situation of weakly incompressible bridge surfaces by Tomova in \cite{To}.

\begin{Lem}[Corollary 6.3, \cite{To}]
\label{Lemma6}
Let $K$ be a knot in a closed, orientable, irreducible 3-manifold $M$.  Let $S_K$ be a weakly incompressible splitting surface for $M_K$ and let $S'_K$ be a surface that is obtained from $S_K$ by compressing $S_K$ to one side.  Then $S'_K$ is incompressible to the other side.
\end{Lem}

\section{Almost normal bridge surfaces}
\label{main}

The proof of the main theorem follows from an application of ideas from \cite{St} where the Heegaard surface is replaced with a bridge surface.  An important difference between the two arguments is that the leaves of the foliations in this case are surfaces with boundary as opposed to closed surfaces.  The argument relies heavily on Lemmas \ref{Lemma5}, \ref{Lemma4}, and \ref{Lemma3} whose statements are close to that of Lemmas 4 and 5 in \cite{St}.  The proofs of these two lemmas are similar in spirit to the originals, but differ in detail.  Several cases missed in the original proof in \cite{St} are included here, more extensive use is made of edgeslides, and the arguments have been adapted to our situation. 

\begin{proof}[Proof of Main Theorem]

Let $K$ be a knot in a closed, orientable 3-manifold $M$ and assume that $M$ and $M_K$ are both irreducible.  Suppose that the knot $K$ is in $n$-bridge position with bridge surface $S$ so that $M=W \cup_{S} W'$ is an irreducible Heegaard splitting of $M$ and the punctured surface $S_K$ is weakly incompressible.  Suppose that $S_K$ separates $M_K$ into the two $K$-handlebodies $W_K$ and $W'_K$.  Let $M_K=W_K \cup_{S_K} W'_K$ denote the bridge splitting of $M$ by $S$.  We can foliate $M_K=W_K \cup_{S_K} W'_K$ with a standard singular foliation that intersects the torus $\partial M_K$ in meridian circles. The top singular leaf of the foliation, $L_{top}$, is a 1-complex given by the union of a spine $\Sigma_{(W,K)}$ of $W_K$ and one meridian circle of $\partial M_K$ for each of the $n$ endpoints of $\Sigma_{(W,K)}$ on $K$.  Similarly, the bottom singular leaf of the foliation, $L_{bot}$, is a 1-complex given by the union of a spine $\Sigma_{(W',K)}$ of $W'_K$ and one meridian circle of $\partial M_K$ for each of the $n$ endpoints of $\Sigma_{(W',K)}$ on $K$.  Thus there is a symmetric picture near the top and bottom leaves of the foliation.  

Consider a nearby leaf of the frontier of a regular neighborhood of $L_{top}$ (resp. $L_{bot}$) in $M_K$.  It can be viewed as consisting of two parts.  The first is a collection $\Gamma^{top}$ (resp. $\Gamma^{bot}$) of $n$ boundary parallel annuli.  Secondly, these annuli are tubed together via the boundary $t^{top}$ (resp. $t^{bot}$) of a regular neighborhood of $\Sigma_{(W,K)}$ (resp. $\Sigma_{(W',K)}$).  Topologically $\Gamma^{top}$ (resp. $\Gamma^{bot}$) consists of $n$ once-punctured annuli and $t^{top}$ (resp. $t^{bot}$) consists of an $n$-punctured copy of $S_K$.  Since $t^{top}$ and $t^{bot}$ arise as the boundary of a regular neighborhood of a 1-complex it is natural to refer to them as collections of ``tubes''.  We will refer to the pair of annuli $\Gamma^{top}$ and $\Gamma^{bot}$ as $\Gamma$.  See Figure \ref{trefoil}.  Throughout the paper the foliation we refer to will always be the standard product foliation of $S^1 \times I$ on each component between the top and bottom annuli on $\partial M_K$. 

\begin{figure}
  \begin{center}
  \includegraphics[width=1.75in]{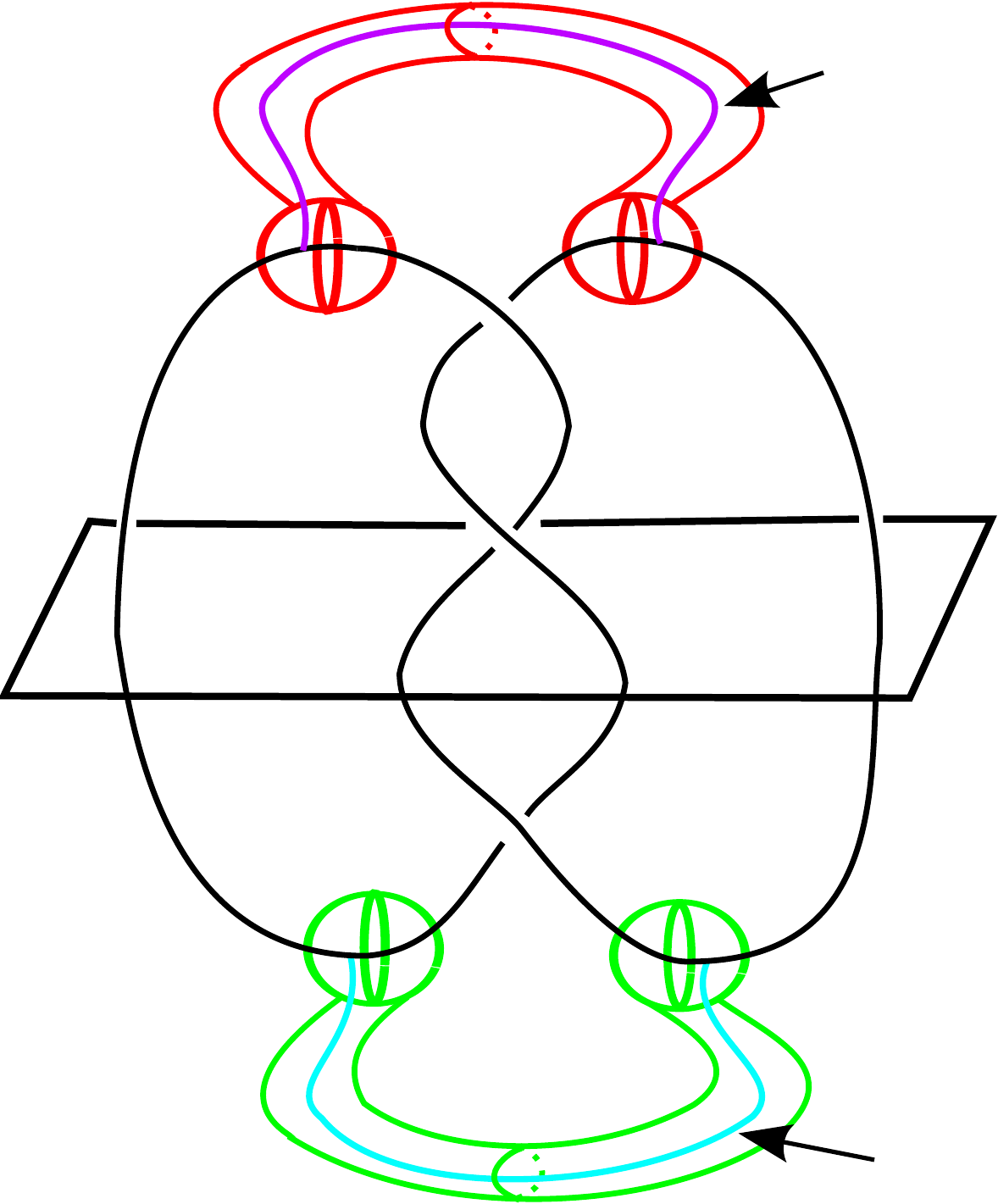}
  \put(-15,34){$K$}
  \put(-18, 141){$\Sigma_{(W,K)}$}
  \put(-118,123){$L_{top}$}
  \put(-115,25){$L_{bot}$}
  \put(-12,2){$\Sigma_{(W',K)}$}
  \caption{An example with $M=S^3$ and $K$ a trefoil}
  \label{trefoil}
  \end{center}
\end{figure}
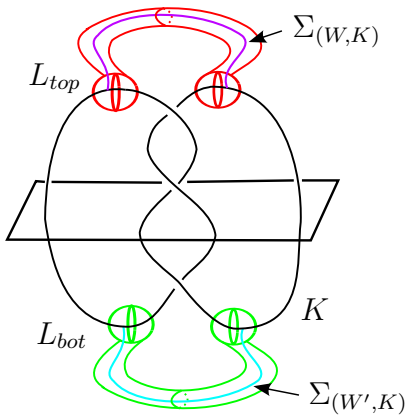

Next we will describe how to triangulate $M_K$ so that the collection $\Gamma$ of annuli are normal with respect to the triangulation.  Consider the collection of $2n$ meridional annuli on $\partial M_K$ parallel to $\Gamma$.  In each annulus choose a meridian circle and view it as the union of a vertex and an edge.  This divides $\partial M_K$ into rectangles.  Now subdivide each rectangle by adding a diagonal edge connecting two adjacent vertices.  This gives a triangulation of $\partial M_K$.  See Figure \ref{triang}. 

\begin{figure}
  \begin{center}
  \includegraphics[width=2.0in]{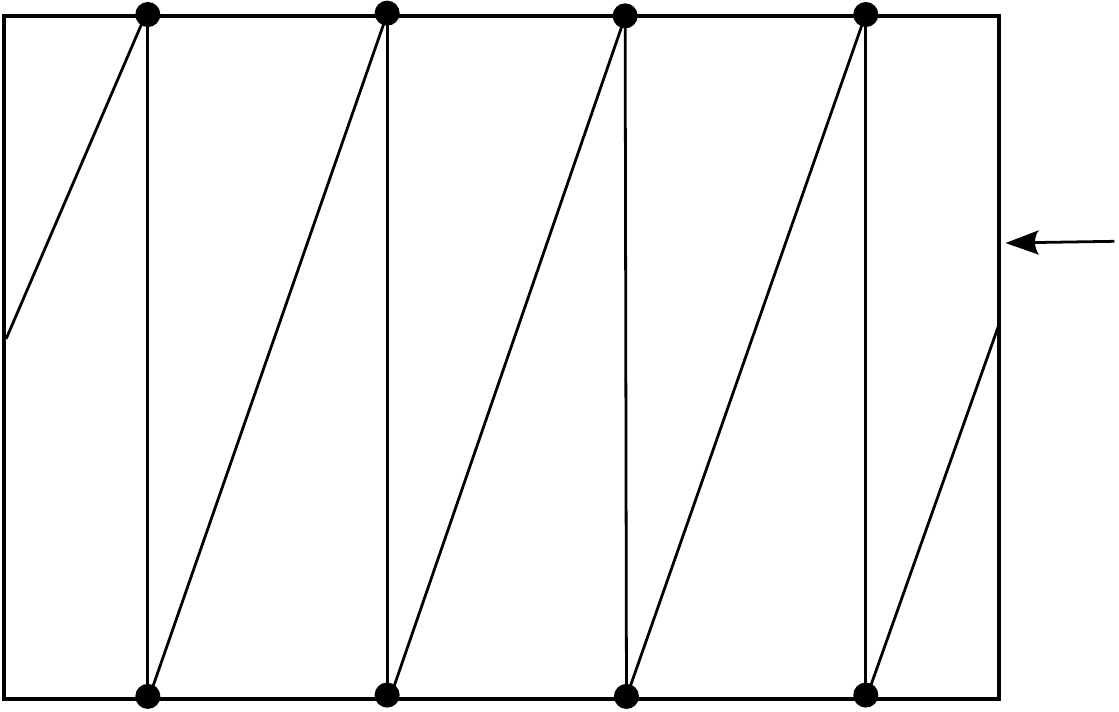}
  \put(-38,6){$\partial M$}
  \put(2,60){\small{meridian}}
  \caption{Triangulation of $\partial M$}
  \label{triang}
  \end{center}
\end{figure}

By \cite{JLR} (also see \cite{JR} pg. 56) we can extend the triangulation of $\partial M_K$ to a triangulation of all of $M_K$ without adding any vertices.  Denote this triangulation of $M_K$ by $\mathcal{T}$.   All of the vertices of this triangulation are contained in a neighborhood of the top and bottom leaves of the foliation.  The collection $\Gamma$ of $2n$ annuli are normal with respect to this triangulation and $\Gamma$ contains all of the $2n$ vertices of the triangulation to one side, separating them from the rest of $M_K$.  Also note that $\mathcal{T}$ has all of its vertices on $\partial M_K$, so it has no vertex-linking 2-spheres.  However, $\mathcal{T}$ may contain normal 2-spheres disjoint from $\Gamma$ that are not vertex-linking.  

Let $\Lambda$ denote a maximal collection of non-parallel disjoint normal 2-spheres in $M_K$ disjoint from $\Gamma$.  Cutting $M_K$ open along $\Lambda$ results in several components, but only one component will contain the torus boundary of $M_K$.  Call this component ${M_0}^+$.  Note that ${M_0}^+$ may have multiple 2-sphere boundary components along with $\partial M_K$.  Since $M_K$ is irreducible each 2-sphere in $\Lambda$ must bound a 3-ball to the opposite side of ${M_0}^+$.  

Since the annuli $\Gamma$ are normal and they contain the vertices of the triangulation to one side, each 2-sphere in $\partial {M_0}^+$ is connected to $\Gamma$ via an edge of $\mathcal{T}^1$.   If an edge of $\mathcal{T}^1$ connects two 2-sphere boundary components of ${M_0}^+$ then by Lemma 2 of \cite{Th} it follows that $M_0$ must be a punctured 3-ball with a torus boundary component, which it clearly is not.  Thus we can conclude each 2-sphere component of $\partial {M_0}^+$ is connected to $\Gamma^{top}$ or $\Gamma^{bot}$ by an edge of $\mathcal{T}^{1} \cap {M_0}^+$.  Assume without loss of generality that an edge connects the 2-sphere to $\Gamma^{top}$.  Taking a tube parallel to this edge connecting $\Gamma^{top}$ to the normal 2-sphere gives an almost normal annulus isotopic to an annulus in $\Gamma^{top}$.  By Lemma \ref{normal} this surface is isotopic to a normal surface giving a new collection of normal annuli that we will call $\Gamma^{top'}$.  We can isotope the tubes $t$ so that they lie in ${M_0}^+$.  We can do this for each normal 2-sphere in $M_0$ and then replace the original singular leaf $L_{top}$ with the singular leaf $L'_{top}=\Gamma^{top'} \cup \Sigma_{(W,K)}$.  Let $M_0$ be the side of $\Gamma^{top'}$ that lies in ${M_0}^+$. 

Let $K_0=K\cap M_0$.  Isotoping the bridge surface $S_K$ to be disjoint from $\Gamma'$ induces a splitting of $M_0$ into two $K_{0}$-compression bodies $W_0$ and $W'_0$.  Continue to call this splitting surfaces $S_K$.  The surface $S_K$ is a weakly incompressible splitting surface for $M_0$.  We can foliate $M_0$ with a standard foliation $F_0$ with leaves isotopic to $S_K$.  The top leaf of the foliation is $L_{top}$ and the bottom leaf is $L_{bot}$.  Let $\mathcal{T}^{1}_0$ denote the part of the 1-skeleton of $\mathcal{T}$ that lies in the interior of $M_0$.  Put $\mathcal{T}^{1}_0$ into thin position with respect to $F_0$.  Let $\Sigma_0$ denote the pair of spines $\Sigma_{(W_0,K_0)}$ and $\Sigma_{(W'_0,K_0)}$ of the $K_0$-compression bodies $W_0$ and $W'_0$ respectively.  Note that $\Sigma_{(W_0,K_0)} \subset \Sigma_{(W,K)}$ and $\Sigma_{(W'_0,K_0)} \subset \Sigma_{(W',K)}$.  If $\mathcal{T}^{1}_0$ intersects $\Sigma_0$ then isotope $\mathcal{T}^{1}_0$ slightly off of $\Sigma_0$. Let $\Gamma_0$ denote the pair $\Gamma^{top'}$ and $\Gamma^{bot'}$.

The triple $(M_0,\Sigma_0, \Gamma_0)$ is the first step in an iterative process.  Each later step will consist of a triple $(M_i,\Sigma_i, \Gamma_i)$ with the following properties:  $M_i \subset M_{i-1}$ will be a submanifold of $M_0$ for each $i$.  The surface $\Gamma_i=\partial M_i-\partial M$ will be a pair $\Gamma_{i}^{top}$ and $\Gamma_{i}^{bot}$ of properly embedded normal surfaces with respect to the triangulation $\mathcal{T}$ given above.  Let $K_i=K \cap M_i$.  The submanifold $M_i$ has a weakly incompressible splitting surface $S_K$ that separates $M_i$ into two $K_i$-compression bodies $W_i$ and $W'_i$.  Let $\Sigma_{i}^{top}$ and $\Sigma_{i}^{bot}$ denote the spines $\Sigma_{(W_i,K_i)}$ and $\Sigma_{(W'_i,K_i)}$ of $K_i$-compression bodies $W_i$ and $W'_i$ respectively.  Let $\Sigma_i$ denote the pair of spines $\Sigma_{i}^{top}$ and $\Sigma_{i}^{bot}$ of $W_i$ and $W'_i$ respectively.  The spine $\Sigma_{i}^{top}$ (resp. $\Sigma_{i}^{bot}$) can be extended to give a spine of $M-\Sigma_{(W',K)}\simeq W$ (resp. $M-\Sigma_{(W,K)}\simeq W'$).  As usual we define $\Sigma_i$ up to isotopy and slides of edges over other edges and over $\partial_{-}{W_i}=\Gamma^{top}_{i}$.  The complement of a regular neighborhood $N(\Sigma_{i}^{top})$ (resp. $N(\Sigma_{i}^{bot}$)) in $M_i$ is foliated by copies of $S_K$ with singular leaf $\Gamma_{i}^{bot} \cup \Sigma_{i}^{bot}$ (resp. $\Gamma_{i}^{top} \cup \Sigma_{i}^{top}$).  Thus $M_i$ can be foliated with a singular foliation $F_i$ by copies of $S_K$ with its top singular leaf consisting of $\Gamma_{i}^{top} \cup \Sigma_{i}^{top}$ and its bottom singular leaf consisting of $\Gamma_{i}^{bot} \cup \Sigma_{i}^{bot}$.  Now consider the edges of $\mathcal{T}^1$ that do not lie on $\partial M_K$ and let $\mathcal{T}^{1}_{i}$ denote their intersection with $M_i$.  Put $\mathcal{T}^{1}_{i}$ into thin position with respect to $F_i$.  

Denote by $t_{i}^{top}$ and $t_{i}^{bot}$ the boundary of a regular neighborhood $N(\Sigma_i)$ in $M_i$ which we continue to call ``tubes''.  A regular leaf near the top (resp. bottom) singular leaf is then obtained by attaching $t_{i}^{top}$ (resp. $t_{i}^{bot}$) to a punctured copy of $\Gamma_{i}^{top}$ (resp. $\Gamma_{i}^{bot}$). 

Here is a sketch of the iterative process that we will describe in detail later.  Start with $(M_i,\Sigma_i, \Gamma_i)$.  If, without loss of generality, $\chi(\Gamma_{i}^{top})=\chi(S_K)$ and $\Gamma_{i}^{top}$ is isotopic to an almost normal surface then since $\Gamma_{i}^{top}$ is obtained by compressing a leaf of the foliation it follows that $\Gamma_{i}^{top}$ is isotopic to $S_K$ and we are done.  Otherwise apply either Lemma \ref{Lemma5} or Lemma \ref{Lemma4} to obtain a new collection of normal and almost normal surfaces in $M_i$.  If there is an almost normal surface $G$ in the collection with $\chi(G)=\chi(S_K)$ then again because the almost normal surface  comes from compressing a leaf of the foliation we know $S$ must be isotopic to $S_K$ and we are done.  Otherwise use Lemma \ref{normal} to isotope the almost normal surfaces (if any) in the collection to be normal.  Then, using this new collection of normal surfaces we can cut $(M_i,\Sigma_i,\Gamma_i)$ along this collection to obtain $(M_{i+1}, \Sigma_{i+1},\Gamma_{i+1})$, which will also satisfy the above properties.  It turns out that we only need to repeat the recursive step a finite number of times before obtaining an almost normal surface isotopic to $S_K$.  This completes the sketch.

 Now, consider the arcs $\mathcal{T}^{1}_i$ in $M_i$ in thin position with respect to $F_i$.  Recall that all ends of $\mathcal{T}^{1}_i$ lie on $\Gamma_{i}^{top}$ or $\Gamma_{i}^{bot}$, part of the top or bottom singular leaves of $F_i$.  One possibility is that there is a maximum of $\mathcal{T}^{1}_i$ that is above a minimum of $\mathcal{T}^{1}_i$ which implies that there is a thick region of $\mathcal{T}^{1}_i$ in $M_i$.  The other possibility is that all of the minima of $\mathcal{T}^{1}_i$ are above all of the maxima of $\mathcal{T}^{1}_i$ and so there is no thick region.  In this situation we will consider separately the following two possibilities.  The first is that there is no thick region and there is some arc of $\mathcal{T}^{1}_i$ with both ends on $\Gamma_{i}^{top}$ or both ends on $\Gamma_{i}^{bot}$.  The second possibility is that there is no thick region and each arc each arc of $\mathcal{T}^{1}_i$ has one endpoint  on $\Gamma_{i}^{top}$ and the other endpoint on $\Gamma_{i}^{bot}$.  We will consider each  of the three possibilities in turn.

The first possibility is that there a thick region of $\mathcal{T}^{1}_i$ with respect to $F_i$.  

\begin{Lem}[cf. Lemma 5, \cite{St}]
\label{Lemma5}
If there is a thick region for $\mathcal{T}^{1}_i$ in $M_i$, then there is a collection of normal and almost normal surfaces in $M_i$ obtained from a leaf of the foliation by compressing the leaf to one side.  At most one surface can be almost normal.  Not all of the surfaces are boundary parallel.
\end{Lem}

\begin{proof}
The proofs of the following Claims \ref{claim4}, \ref{claim1}, \ref{claim2}, and \ref{claim3} can be found in \cite{Th}.  
Let $(M_i,\Sigma_i, \Gamma_i)$ be as described above, where $\mathcal{T}^{1}_i$ is in thin position with respect to the foliation $F_i$ of $M_i$.  Since there is a thick region of $\mathcal{T}^{1}_i$ in $F_i$ we can apply Claim 4.5 of \cite{Th}. 

\begin{Claim}[Claim 4.5, \cite{Th}]
\label{claim4} There exists a transverse leaf $L$ in the first thick region of $F_i$ which intersects the 2-skeleton entirely in normal arcs and simple closed curves disjoint from the 1-skeleton.  
\end{Claim}

Let $L$ be a leaf of $F_i$ in a thick region intersecting the 2-skeleton in normal arcs and simple closed curves disjoint from $\mathcal{T}^1$ as is guaranteed by Claim \ref{claim4}.  Then we can apply the following Claims \ref{claim1}, \ref{claim2}, and \ref{claim3} to the leaf $L$.  

\begin{Claim}[Claim 4.1, \cite{Th}]
\label{claim1}
  Let $H$ be any tetrahedron in the triangulation $\mathcal{T}$ of $M_K$.  Then $L \cap H$ contains no parallel curves of length greater than or equal to eight.
\end{Claim}

\begin{Claim}[Claim 4.2, \cite{Th}] 
\label{claim2} 
Let $H$ be any tetrahedron in the triangulation $\mathcal{T}$ of $M_K$.  Then $L \cap \partial H$ contains no curve of length greater than eight.
\end{Claim}

\begin{Claim}[Claim 4.3, \cite{Th}]  
\label{claim3}Let $H_1$ and $H_2$ be distinct tetrahedra in the triangulation of $M_K$.  Then $L \cap \partial {H_1}$ and $L \cap \partial {H_2}$ do not both contain curves of length eight. 
\end{Claim}

The above claims together imply that this leaf $L$ of the foliation $F_i$ intersects the 2-skeleton only in simple closed curves disjoint from the 1-skeleton and normal curves of lengths three, four, and at most one of length eight.  Compressing the simple closed curves in $\mathcal{T}^2$ and the surfaces in the interior of the tetrahedra gives a collection of normal surfaces with at most one almost normal surface.  The almost normal surface, if it exists, must be a normal octagon since we have compressed all annuli in the interior of the tetrahedra.  We can think of the leaf $L$ as this collection of normal and almost normal surfaces with tubes attached.  Since our triangulation has no normal 2-spheres we can conclude that this collection will contain no almost normal 2-spheres as well since any almost normal 2-sphere can be isotoped to give a normal one by Lemma \ref{normal}.  Also notice that in this collection we will not have a normal surface tubed to the opposite side of another normal surface or we get a contradiction to the leaf being weakly compressible.  To see that all of the compressions of $L$ are to one side, observe that as long as there are no normal 2-spheres in the collection then we know that compressions of $L \cap H$ are compressions of $L$ for any tetrahedron $H$ of $\mathcal{T}$. Lemma \ref{Lemma6} implies that after compressing $L$ to one side the remaining surface is incompressible to the opposite side, thus there cannot be compressions on opposite sides of $L$.   This completes the proof of Lemma \ref{Lemma5}.
\end{proof}

\begin{Rem}  There is no choice in the direction in which the tubes of $L$ compress, however Theorem \ref{Lemma6} implies that after compressing $L$, the remaining collection of normal and almost normal surfaces is incompressible in the direction opposite to which we have compressed. 
\end{Rem}

For the proof of the Lemma \ref{Lemma4} we will need the following:

\begin{Claim}
\label{claim}
A properly embedded, orientable, normal surface is incompressible and boundary incompressible in the complement of the 1-skeleton $\mathcal{T}^1$.  
\end{Claim}

\begin{proof}
This claim follows from a standard innermost disk and outermost arc argument.  
\end{proof}

The second possibility is that there is no thick region and some arc of $\mathcal{T}^{1}_{i}$ has both endpoints on $\Gamma_{i}^{top}$ or both endpoints on $\Gamma_{i}^{bot}$.

\begin{Lem}[cf. Lemma 4, \cite{St}]
\label{Lemma4}
If there is no thick region for $\mathcal{T}^{1}_i$ in $M_i$ and some arc of $\mathcal{T}^{1}_i$ has both endpoints on $\Gamma_{i}^{top}$  (resp. has both endpoints on $\Gamma_{i}^{bot}$), then there is an almost normal surface in $M_i$ that is isotopic to a surface obtained from a leaf of $F_i$ by compressing the leaf above (resp. below). 
\end{Lem}

\begin{proof}
We will prove the Lemma for arcs of $\mathcal{T}^{1}_i$ with both endpoints on $\Gamma_{i}^{top}$.  The argument for arcs of $\mathcal{T}^{1}_i$ with both endpoints on $\Gamma_{i}^{bot}$ is symmetric.  Let $L$ be a leaf of the foliation near the top singular leaf, so that $L$ consists of the normal surface $\Gamma_{i}^{top}$ punctured and attached to the tubes $t_{i}^{top}=\partial N(\Sigma_{i}^{top})$.  See Figure \ref{trefoil} for the case $i=0$ and $M=S^3$.  Choose an arc $\beta$, a subarc of $\mathcal{T}^{1}_{i}$.  Since there is no thick region for $\mathcal{T}^{1}_i$, $\beta$ has only a single minimum and it is parallel to an arc on $L$, so there is a lower disk $E$ whose boundary is the union of $\beta$ and an arc $\alpha$ in $L$.  We will show that after some edge slides and isotopies of $\Sigma_{i}^{top}$, $\alpha$ runs once over exactly one tube of $t_{i}^{top}$, and that this tube connects two normal disks in a tetrahedron, therefore is part of an almost normal surface.  

Define the complexity of $E$ to be $(a,b)$, lexicographically ordered, where $a$ is the number of points of $\Sigma_{i}^{top} \cap \mathcal{T}^2$ to which $\alpha$ is also incident, and $b$ is the number of components in which $E$ meets the 2-skeleton of $\mathcal{T}$.  We will assume that the complexity of $E$ has been minimized over all choices of $E$. 

Observe that the arc $\alpha$ of $\partial E$ can't lie entirely in $\Gamma_{i}^{top}$.  Otherwise $E$ would be a boundary compressing disk in the complement of the 1-skeleton of the normal surface $\Gamma_{i}^{top}$ which contradicts Claim \ref{claim}.  Our strategy will be to show that there is a sequence of proper isotopies of $\Sigma_{i}^{top}$ and slides of ends of arcs of $\Sigma_{i}^{top}$ over each other and over $\Gamma_{i}^{top}$ (neither of which affect the isotopy class of $\Gamma_{i}^{top} \cup \Sigma_{i}^{top}$) so that afterwards $\alpha$ is incident to a single edge $z$ in $\Sigma_{i}^{top}$, $\alpha$ runs along this arc once, and $E$ lies entirely inside a single tetrahedron.  Then $\Gamma_{i}^{top} \cup \partial (N(z))$ is the required almost normal surface obtained from $L$ by compressing all other tubes of $\Sigma_{i}^{top}$.  Now that we have established some notation, for the rest of the proof we will consider the intersections of the disk $E$ with the 2-skeleton of $\mathcal{T}$ and we will show that any intersection violates the minimality of $(a,b)$  

When we consider the arcs of intersection between $E$ and the 2-skeleton, we can get four types of components of $E \cap \mathcal{T}^2$ in $E$.  Components of Type I are simple closed curves in $E$.  Components of Type II are arcs with both endpoints on $\alpha$.  Components of Type III are arcs with both endpoints on $\beta$, and components of Type IV are arcs with one endpoint on $\alpha$ and the other endpoint on $\beta$.  See Figure \ref{diskE}.  Next we will describe how each type of component of intersection of $E\cap \mathcal{T}^2$ can be removed, violating minimality of $(a,b)$.   

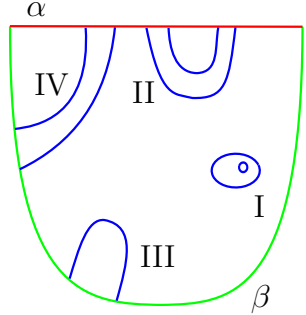
\begin{figure}
  \begin{center}
  \includegraphics[width=1.55in]{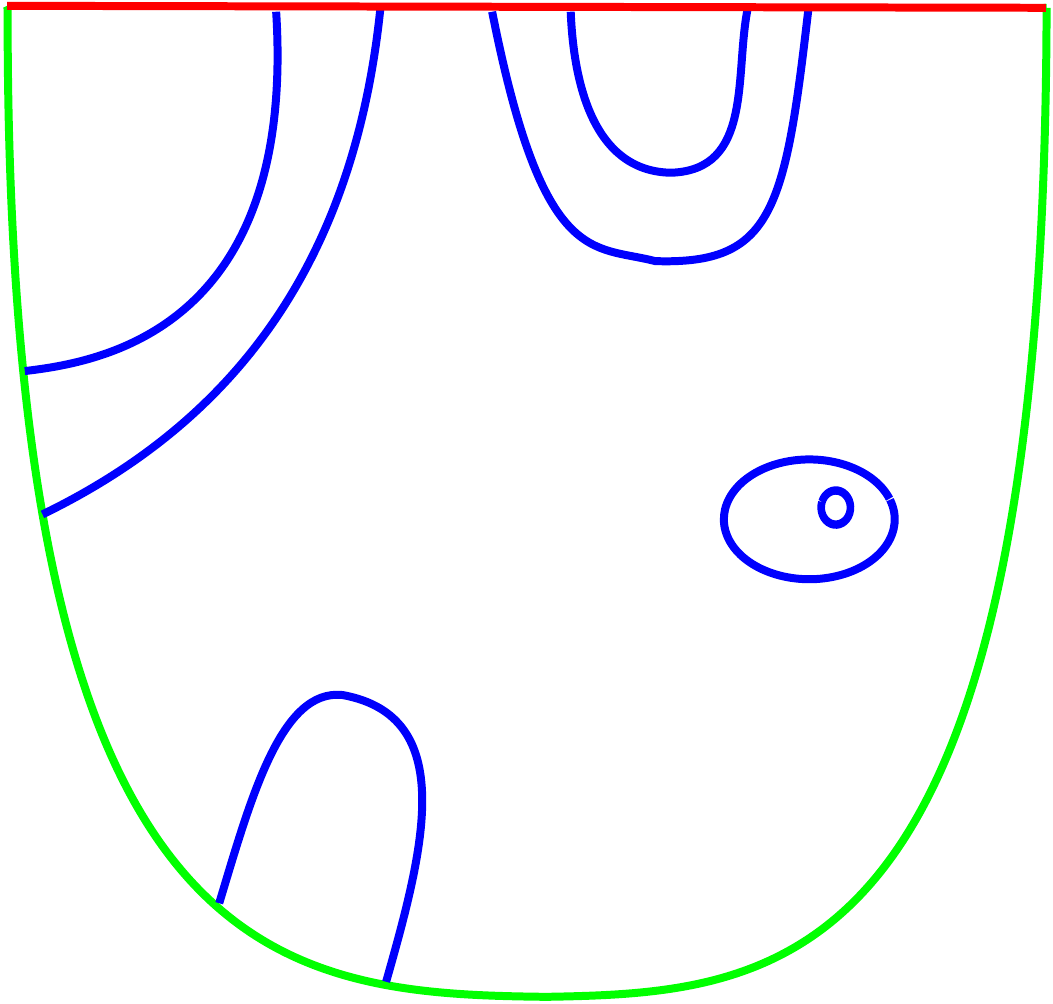}
 \put(-20,0){$\beta$}
 \put(-105, 110){$\alpha$}
 \put(-102, 80){IV}
 \put(-65, 77){II}
 \put(-19,34){I}
 \put(-62, 15){III}
  \caption{The disk E}
  \label{diskE}
  \end{center}
\end{figure}

Components of Type I are simple closed curves in $E$.  A component of Type I that is innermost in a 2-simplex of $\mathcal{T}^2$ can be removed by substituting the disk it bounds in $\mathcal{T}^2$ for the disk it bounds in $E$.  This reduces the number of times that $E$ meets the 2-skeleton, thus reducing $b$ and contradicting the minimality of $E$.

A component of Type II corresponds to an arc in a face $\sigma$ of some tetrahedron of $\mathcal{T}$ that either 1) has both endpoints on distinct components of $(\Sigma_{i}^{top} \cup \Gamma_{i}^{top})\cap \sigma$, or 2) has both endpoints on the same component of $(\Sigma_{i}^{top} \cup \Gamma_{i}^{top})\cap \sigma$.  Suppose $\gamma$ is an arc of intersection between $\mathcal{T}^2$ and $E$ that is outermost in $E$ and is of Type II.  Let $\delta$ be a subarc of $\alpha$ such that $\gamma \cup \delta$ is the boundary of a disk $D$ in $E-\mathcal{T}^2$.  See Figure \ref{type2-type4}b).  Then the arc $\gamma$ is either of type 1) or 2) above.  In what follows we will use edge slides of $\Sigma_{i}^{top}$ to remove components of $\mathcal{T}^2 \cap E$ and $\mathcal{T}^2 \cap \Sigma_{i}^{top}$.  Recall that $t_{i}^{top}$ is the boundary of a neighborhood of $\Sigma_{i}^{top}$, and we will abuse notation and consider $\delta \subset \alpha$ as an arc on $\Sigma_{i}^{top}$ when we really mean that $\delta$ is an arc on $t_{i}^{top}$. 

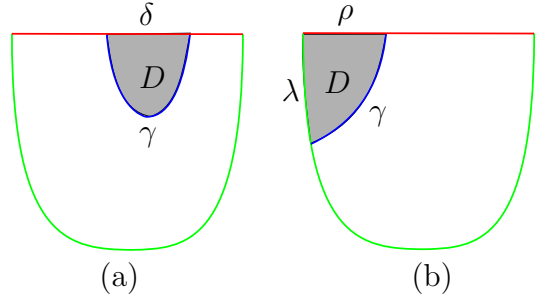
\begin{figure}
  \begin{center}
  \includegraphics[width=2.75in]{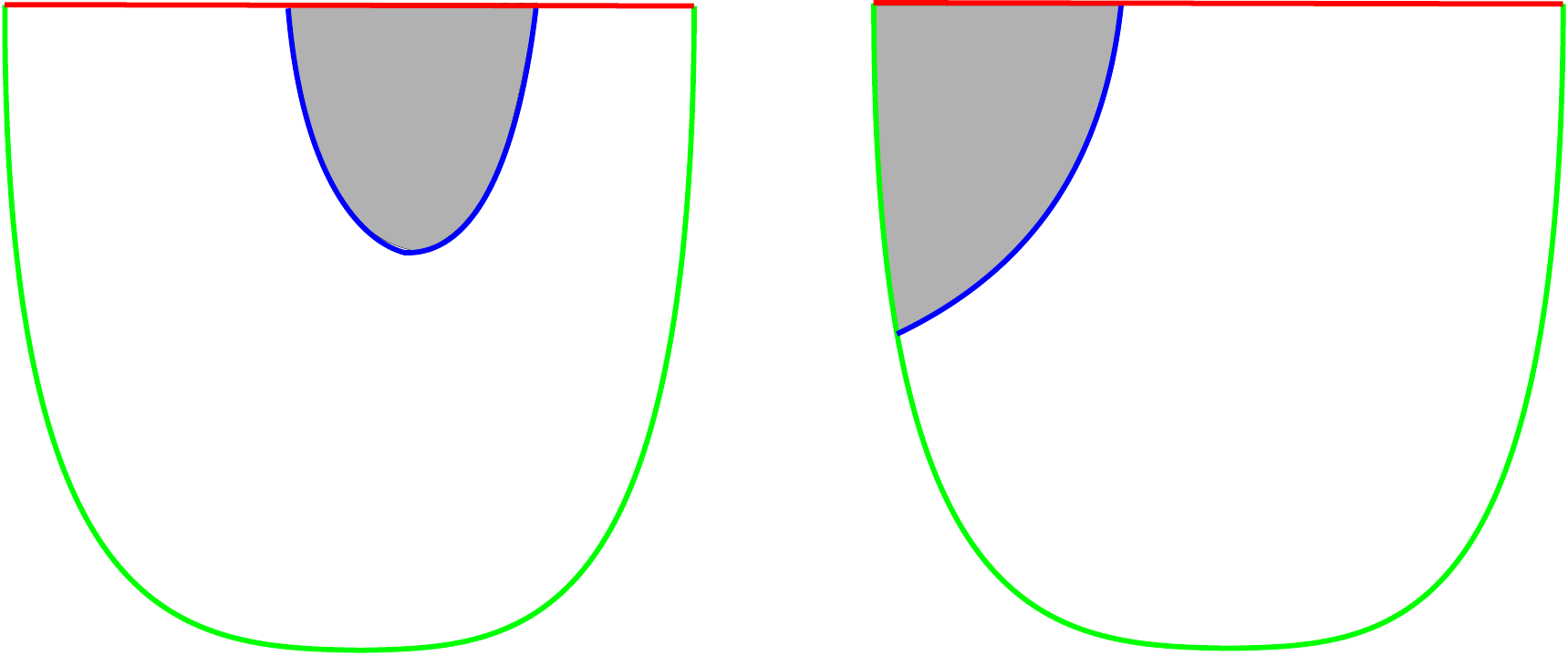}
 \put(-150,86){$\delta$}
 \put(-150,42){$\gamma$}
 \put(-75, 88){$\rho$}
 \put(-96, 57){$\lambda$}
 \put(-63,50){$\gamma$}
 \put(-80,60){$D$}
 \put(-150,62){$D$}
 \put(-165,-13){(a)}
 \put(-47,-13){(b)}
 \caption{Arcs of Type II and IV in $E$}
  \label{type2-type4}
  \end{center}
\end{figure}

Any two ends of edges of $\Sigma_{i}^{top}$ that meet the same normal disk in $\Gamma_{i}^{top}$ can be isotoped together so that there is at most one edge incident to each normal disk.  Since $\Sigma_{i}^{top}$ can be extended to give a spine of $W$, any cycle in $\Sigma_{i}^{top}$ gives a cycle in $\Sigma$.  Since $W \cup_{S} W'$ is an irreducible Heegaard splitting of $M$ it follows from \cite{F} (also see Proposition 2.5 of \cite{STh}) that no cycle of $\Sigma_{(W,K)}$ lies in a 3-ball.  Hence for any tetrahedron $H$ of $\mathcal{T}$,  $\Sigma_{i}^{top} \cap H$ cannot contain a cycle.  Thus $\Sigma_{i}^{top} \cap H$ is a union of trees for each tetrahedron $H$ in $\mathcal{T}$.  Each component of $\Sigma_{i}^{top} \cap H$ is a tree and each component of $(\Sigma_{i}^{top} \cup \Gamma_{i}^{top}) \cap H$ is a tree with disks attached and so is simply connected.

Arcs of type 1) fall into the  following subcases:

\textbf{Case a)} Let $H$ be the tetrahedron in $\mathcal{T}$ that contains $\delta$ and let $q$ denote the component of $\Sigma_{i}^{top} \cap H$ that contains $\delta$.  If $q$ is a single arc with both endpoints on the same face of $H$ then $D$ describes an isotopy that removes two points of intersection of $q$ with $\mathcal{T}^2$.  This reduces the number of points of $\Sigma_{i}^{top} \cap \mathcal{T}^2$ to which $\alpha$ is incident, thus reducing $a$, which is a contradiction. 

\textbf{Case b)}  Now suppose that $q$ is not an arc and $\delta$ is a path in $q$ that begins at point $x$ in $q$ that is not in a normal disk but is in some face $\sigma$ of $H$.  Let $z$ be the edge of $\Sigma_{i}^{top}$ containing $x$.  Then $\delta$ describes a series of edge slides of $z$ which culminate by introducing an extra point of intersection between $\Sigma_{i}^{top}$ and $\sigma$.  However, after the edge slides the disk $D$ runs only over the edge $z$.  Hence we can reduce the number of intersections of edges of $\Sigma_{i}^{top}$ incident to $\alpha$ that meet $\mathcal{T}^2$ by two as in Case a) reducing $a$ by at least one and contradicting the minimality assumptions.  See Figure \ref{edgeslide1}.

\begin{figure}
  \begin{center}
  \includegraphics[width=4.5in]{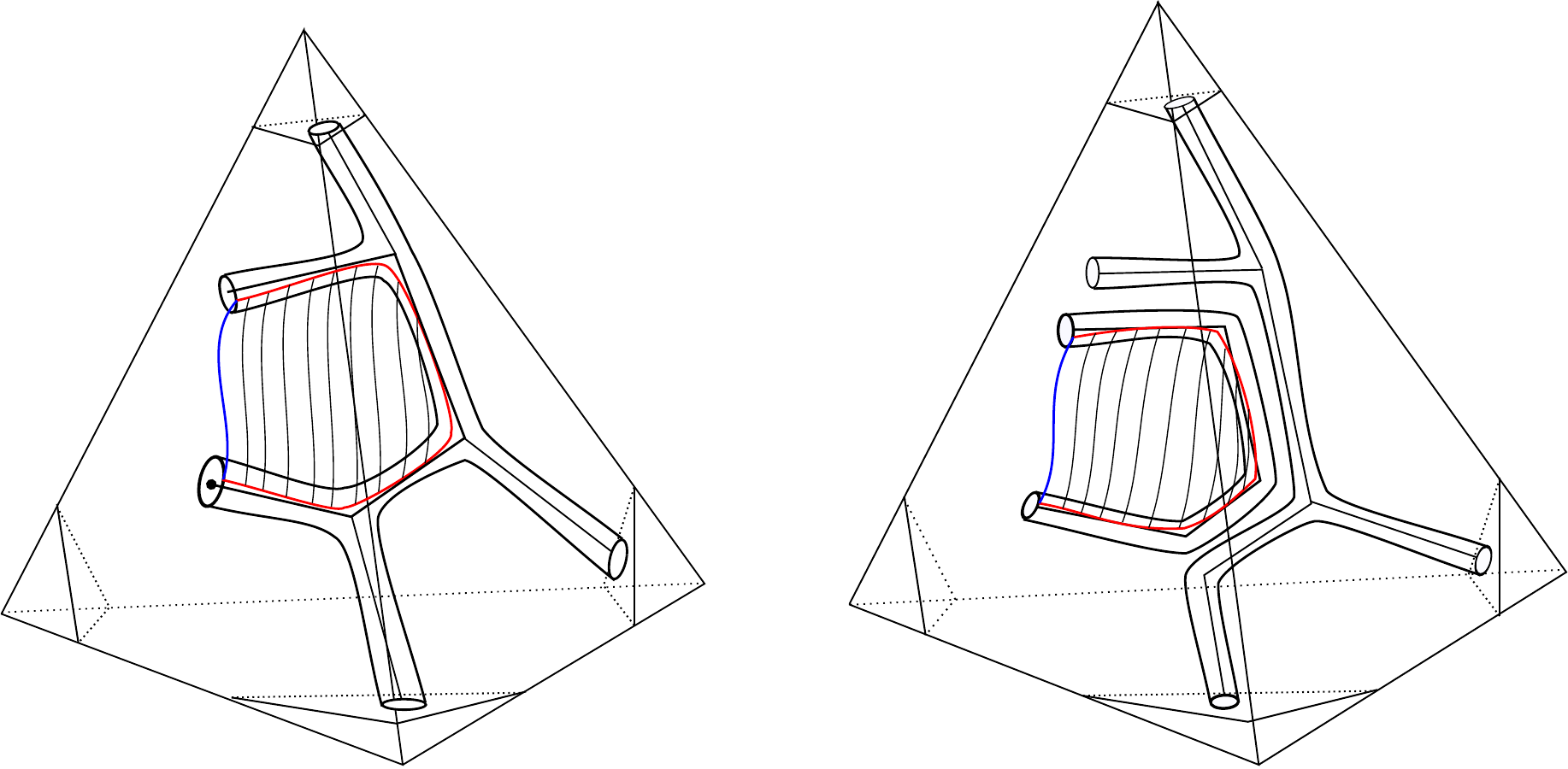}
  \put(-292,58){$x$}
  \put(-289,77){$\gamma$}
  \put(-265,75){$D$}
  \put(-270, 44){$z$}
  \put(-96, 66){$D$}
  \caption{}
  \label{edgeslide1}
  \end{center}
\end{figure}

\textbf{Case c)}  If $\delta$ is an arc on $L$ that has both endpoints of $\mathcal{T}^2 \cap E$ on normal disks of $L$, then either $\delta$ must run over some edges of $\Sigma_{i}^{top}$ or it lies in a normal disk.  If $\delta$ lies in a normal disk then as an outermost arc of the normal disk it cuts off a subdisk $D$ in the normal disk.  Together $\gamma$ and $\delta$ bound a disk $D'$ so that $D\cup D'$ bounds a 3-ball that can be used to isotope $\delta$ into the next tetrahedron removing $\gamma$ and thus reducing complexity.  

So we can assume that $\delta$ runs over some edges of $\Sigma_{i}^{top}$. Say that $\delta$ runs from normal disk $D_1$ to normal disk $D_2$.  Since $\Sigma_{i}^{top}$ is incident to $D_1$ in only a single point, $\delta$ is incident to $\partial D_1$ in a single point.  It follows that $D_1-N(\Sigma_{i}^{top})$ is an annulus and $\delta$ intersects the annulus in a single spanning arc.  Thus $\delta$ runs precisely once along the edge $z$ that is incident to $D_1$.  Then, as above, $D$ describes a slide and isotopy of $z$ that carries it to the arc $\gamma$ in a simplex of $\mathcal{T}^2$.  But then a subdisk of that face describes a parallelism between $z$ and a subarc of $\mathcal{T}^1$.  In particular, attaching a tube to $\Gamma_{i}^{top}$ along $z$ gives an almost normal surface. 

Arcs of type 2) have endpoints on the same component of $q\cap \partial H$.  Let $x$ denote the endpoint of $q$ in the face $\sigma$ of tetrahedron $H$.  We assume that $x\in \Sigma_{i}^{top}$, so $\gamma$ forms a loop based at $x$ in $\sigma$ bounding a disk $A$ in $\sigma$; the case where both ends of $\gamma$ lie on a normal disk is similar:  

\textbf{Case a)}  If $interior(A)\cap \Sigma_{i}^{top}=\emptyset$, then construct new disks $E'$ and $E''$ by cutting the disk $D$ along $\gamma$ and attaching a copy of $A$ to each piece.  One of the disks $E'$ or $E''$ will still be a lower disk and it will meet $\mathcal{T}^2$ in fewer components than $E$, contradicting the minimality of $E$.

\textbf{Case b)}  Now suppose that $A\cap \Sigma_{i}^{top} \neq \emptyset$.  Since $\Sigma_{i}^{top}$ is a union of trees in $H$, we know that a neighborhood of each component $q$ of $\Sigma_{i}^{top} \cap H$ is a 3-ball.  So there is a disk $D'$ in $N(q)$ whose boundary is the union of $\delta$ and a diameter $\delta'$ of a small disk $\epsilon$ with which $N(q)$ meets $\partial H$ at $x$.  

Isotope the leaf $L$ by compressing $\delta$ to $\delta'$ in $N(q)$, splitting the disk $\epsilon$ in two.  The effect on the spine is a possibly complicated series of edge slides.  The overall effect is that the number of components of $\Sigma_{i}^{top} \cap \sigma$ increases by one when $\epsilon$ splits, and $D\cup D'$ becomes a disk disjoint from $\Sigma_{i}^{top}$ and parallel to $A$.  The disk $A$ now contains at least two points of $\Sigma_{i}^{top} \cap \sigma$.  Now push $D\cup D'$ across $A$ to remove $\gamma$, thus reducing $b$ which is a contradiction.  See Figure \ref{edgeslides2}.

\begin{figure}
  \begin{center}
  \includegraphics[width=4.4in]{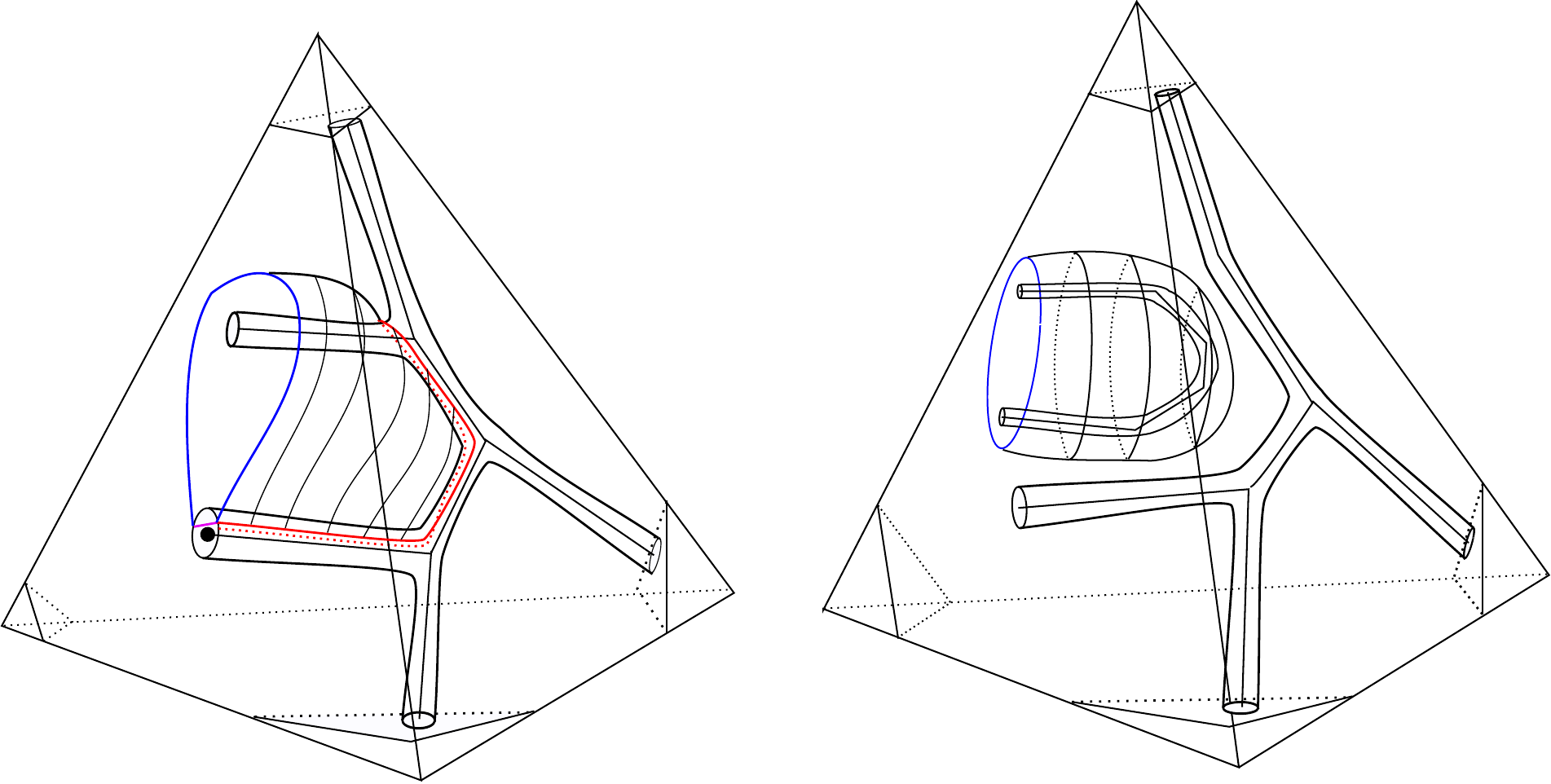}
  \put(-273,38){$x$}
  \put(-275,75){$A$}
  \put(-288, 67){$\gamma$}
  \put(-255, 60){$D$}
  \caption{}
  \label{edgeslides2}
  \end{center}
\end{figure}

To see how to remove components of Types III and IV it will be helpful to view the arc $\beta$ that runs along the edge $e$ of $\mathcal{T}^1$ as an arc that lies on $\partial N(e)$.  As an arc on $\partial N(e)$,  $\beta$ may wind around the edge $e$.  If the winding is not monotone a priori then we can reduce the number of components in which the disk $E$ meets the faces of $\mathcal{T}^2$, contradicting minimality.  Thus we may assume that the curve $\beta$ winds monotonically around the edge $e$.  This implies that there are no curves of Type III since the existence a curve of Type III means that the arc $\beta$ must `double back' as it winds around $e$, contradicting monotonicity.  

Let $\gamma$ be an outermost arc component of Type IV, and let $D$ be the corresponding outermost sub-disk of $E$.  Let $\sigma$ be the face of $\mathcal{T}^2$ that contains $\gamma$.  The disk $D$ is co-bounded by a sub-arc $\rho$ of $\alpha$, a sub-arc $\lambda$ of $\beta$, and $\gamma$.  See Figure \ref{type2-type4}a).  Let $H$ denote the tetrahedron containing $D$ in its interior and with $\sigma$ as a face.  

There are two cases that we will consider separately.  The first case is when the arc $\gamma$ of $E\cap \sigma$ that runs from the edge $e$ to $L$ ends on a normal disk $\eta$ of $L$.  The second case is when $\gamma$ ends on a tube (neighborhood of $\Sigma_{i}^{top}$) of $L$.  

\textbf{Case a)}:  Suppose first that $\gamma$ ends on a normal disk $\eta$.  In this situation there are two subcases.  Either $\Sigma_{i}^{top} \cap \rho=\emptyset$ or $\Sigma_{i}^{top}\cap \rho \neq \emptyset$.  See Figure \ref{figure6}a) and \ref{figure6}b).  

\textbf{Subcase 1)}: Suppose that $\Sigma_{i}^{top}\cap \rho=\emptyset$.  In this case the arc $\rho$ runs over only normal disks.  Observe that there is a disk $D'$ in $\partial N(e)$ that is bounded by $\lambda$, a copy of part of the edge $e$ that bounds the face $\sigma$, and a copy of a meridian of $\partial N(e)$.  In this case $D'\cup \sigma \cup \eta$ bounds a 3-ball in $H$ that we can use to isotope $D$ across $\sigma$ and into the next tetrahedron, removing $\gamma$ and reducing $b$, thus reducing the complexity of $E$, a contradiction.  See Figure \ref{figure6}a).  

\textbf{Subcase 2)}:  Suppose now that $\Sigma_{i}^{top} \cap \rho \neq \emptyset$.  Since $\Sigma_{i}^{top}$ is a union of trees, each component of $N(\Sigma_{i}^{top})$ is a 3-ball.  In particular, there is a disk $\Delta$ in $N(\Sigma_{i}^{top})$ whose boundary is the union of a sub-arc of $\rho$ and a diameter $d$ of the disk $\epsilon$ with which $N(\Sigma_{i}^{top})$ intersects the normal disk $\eta$.  Isotope $N(\Sigma_{i}^{top})$ by compressing $d$ to $\rho$ in $N(\Sigma_{i}^{top})$, splitting the disk $\epsilon$ in two.  The effect on $\Sigma_{i}^{top}$ is a series of edge slides that results in a new component of $\Sigma_{i}^{top} \cap H$ that is on the same side of $\rho$ as $e$.  Now proceed as in Subcase 1).  See Figure \ref{figure6}b). 

\begin{figure}
  \begin{center}
  \includegraphics[width=4in]{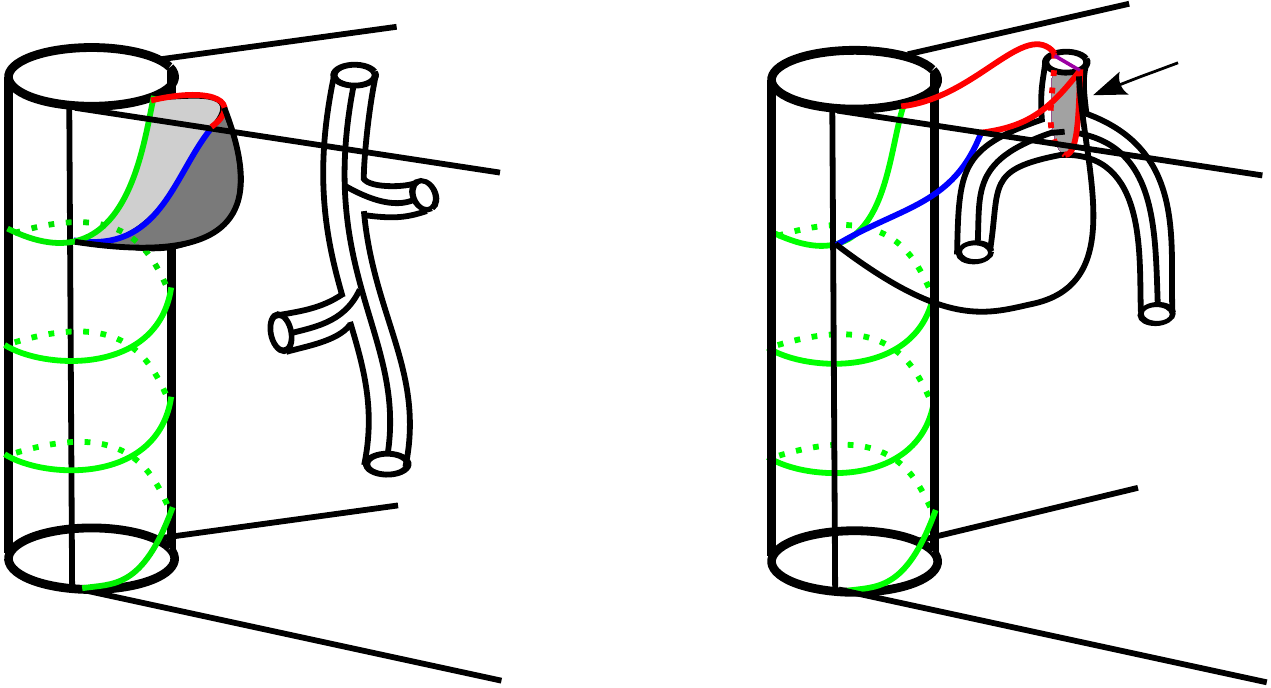}
  \put(-103,4){$N(e)$}
  \put(-65,-15){(b)}
  \put(-282,5){$N(e)$}
  \put(-241,-15){(a)}
  \put(-271,118){\tiny{$D'$}}
   \put(-97,118){\tiny{$D'$}}
  \put(-243,93){\tiny{$D$}}
   \put(-45,82){\tiny{$D$}}
    \put(-191,52){$\Sigma_i$}
    \put(-18,80){$\Sigma_i$}
    \put(-40,143){$\epsilon$}
    \put(-18,140){\small{$\Delta$}}
  \caption{Arcs of type IV}
  \label{figure6}
  \end{center}
\end{figure}
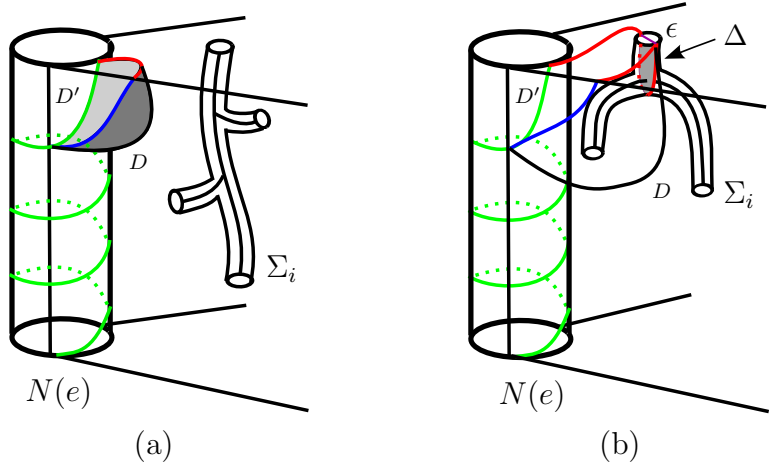

\textbf{Case b)}:  Now suppose that $\gamma$ ends on a tube of $L$.  The core of this tube is an edge $\tau$ that may connect to other edges of $\Sigma_{i}^{top}$ in $\Sigma_{i}^{top} \cap H$, and $\Sigma_{i}^{top}$ connects to a normal disk $\eta$.  See Figure \ref{figure7}.  We will describe in two steps a slide of $\tau$ and an isotopy of $D$ that will remove a component of $\Sigma_{i}^{top}$, reducing $a$, and thereby reducing the complexity of $E$.  

First, since $\tau$ connects to other edges of $\Sigma_i$ in $\Sigma_i\cap H$, $\rho$ describes an edge slide of $\tau$ that keeps $\tau \cap \sigma$ fixed but slides the opposite end of the edge $\tau$ off of $\Sigma_i$ and onto the normal disk $\eta$.  We continue to slide $\tau$ along $\eta$ following $\rho$ until it almost meets $\partial N(e)$.  See Figure \ref{figure7}b).  Now we can use the disk $D$ to isotope all of $\tau$ until it lies close to $\lambda \cup \gamma$.  At this point the entire disk $D$ and tube $\tau$ lie very close to $\beta \cup \gamma$ in $H$.  See Figure \ref{figure7}c).

For the second step recall the disk $D'$ in $\partial N(e)$ that is bounded by $\lambda$, a copy of part of the edge $e$ that bounds the face $\sigma$, and a copy of a meridian of $\partial N(e)$.  The disks $D$ and $D'$ describe an isotopy of $\tau$ across the face $\sigma$ and into the next tetrahedron, removing the component $\gamma$ from $\sigma \cap E$ and, in particular, removing the component of intersection between $\tau$ and $\sigma \in \mathcal{T}^2$, reducing  $a$, which is a contradiction.  See Figure \ref{figure7}d).  Thus there can be no arcs of Type IV.  Therefore the arc $\beta$ of $\mathcal{T}^{1}_i$, the edges of $\Sigma_{i}^{top}$ that $\alpha$ runs along, and the disk $E$ are all contained in one tetrahedron.  The proof now follows as in Case 1c) above.
\end{proof}

\begin{figure}
  \begin{center}
  \includegraphics[width=4.25in]{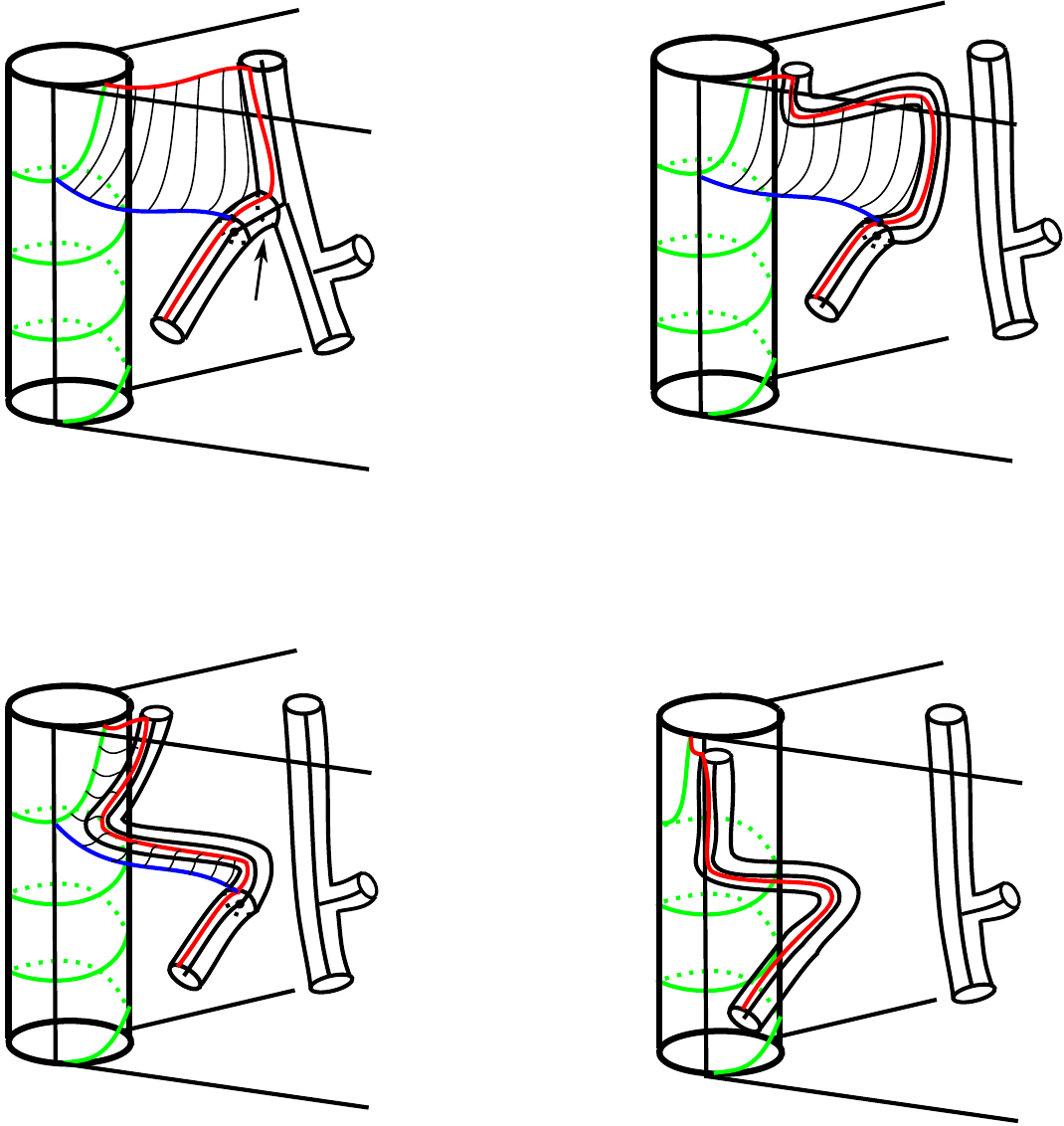}
  \put(-116,0){$N(e)$}
  \put(-75,-20){(d)}
  \put(-307,3){$N(e)$}
  \put(-265,-15){(c)}
  \put(-302,188){$N(e)$}
  \put(-265,175){(a)}
  \put(-238,230){$\tau$}
  \put(-118,192){$N(e)$}
  \put(-75,175){(b)}
  \put(-201,229){$\Sigma_i$}
  \put(-290,103){\tiny{$D'$}}
  \put(-257,278){$D$}
  \put(-62,275){$D$}
  \caption{Arcs of type IV}
  \label{figure7}
  \end{center}
\end{figure}
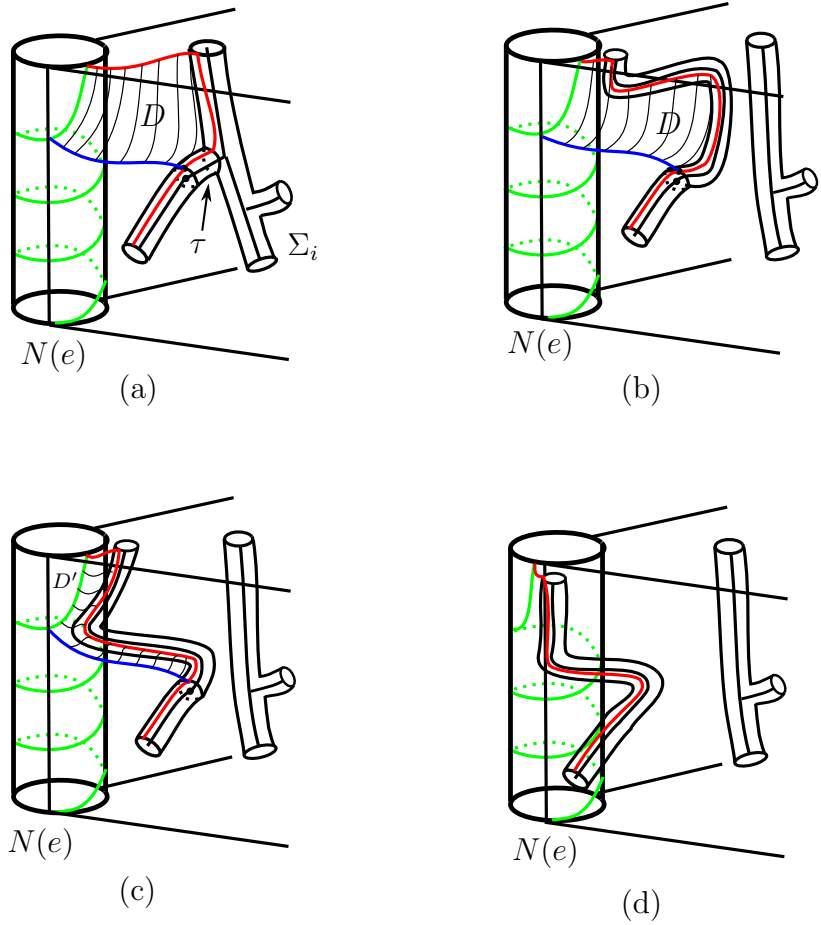

The third possibility is that there is no thick region and each arc of $\mathcal{T}^{1}_i$ has one endpoint on $\Gamma_{i}^{top}$ and the other endpoint on $\Gamma_{i}^{bot}$.

\begin{Lem}
\label{Lemma3}
If there is no thick region for $\mathcal{T}^{1}_i$ in $M_i$ and each arc of $\mathcal{T}^{1}_i$ has one endpoint on $\Gamma_{i}^{top}$ and has the other endpoint on $\Gamma_{i}^{bot}$ then $M_i$ is a product region. 
\end{Lem}

\begin{proof}
Recall that $\partial M_i=\Gamma_{i}^{top} \cup \Gamma_{i}^{bot}$.  Since $\Gamma_{i}^{top}$ and $\Gamma_{i}^{bot}$ are normal with respect to $\mathcal{T}$ it follows that there are two possibilities for how the region $M_i$ between $\Gamma_{i}^{top}$ and $\Gamma_{i}^{bot}$ can intersect a face of the 2-skeleton.  Either the region bounded by $\Gamma_{i}^{top}$ and $\Gamma_{i}^{bot}$ is a trapezoid region or a hexagon region.

Suppose that there is a hexagon region of $M_i\cap \mathcal{T}^2$.  Then three edges of the hexagon are arcs of $(\Gamma_{i}^{top} \cup \Gamma_{i}^{bot}) \cap \mathcal{T}^2$ and that the other three edges are arcs of $\mathcal{T}^{1}_{i}$ connecting the three components of $\Gamma_i \cap \mathcal{T}^2$.  But this implies that some arc of $\mathcal{T}^{1}_{i}$ connects either $\Gamma_{i}^{top}$ to  $\Gamma_{i}^{top}$ or $\Gamma_{i}^{bot}$ to $\Gamma_{i}^{bot}$ which is a contradiction.  Therefore there cannot be any hexagonal components and all regions of intersection between $M_i$ and $\mathcal{T}^2$ are trapezoids.  

Because we know that there are no hexagonal components of intersection between $M_i$ and $\mathcal{T}^2$ this implies that the only possibilities for components of intersection between $M_i$ and the tetrahedra of the 3-skeleton are triangular product regions and quadrilateral product regions.  See Figure \ref{Figure9}.  Each triangular and quadrilateral product region is bounded on one side by a normal disk of $\Gamma_{i}^{top}$ and on the other by a normal disk of $\Gamma_{i}^{bot}$.  Since each component of $M_i\cap H$ is a product region with one end on each of $\Gamma_{i}^{top}$ and $\Gamma_{i}^{bot}$ for each tetrahedron $H$ in $\mathcal{T}$ we can conclude that $M_i$ is itself such a product region.

\begin{figure}
  \begin{center}
  \includegraphics[width=2.45in]{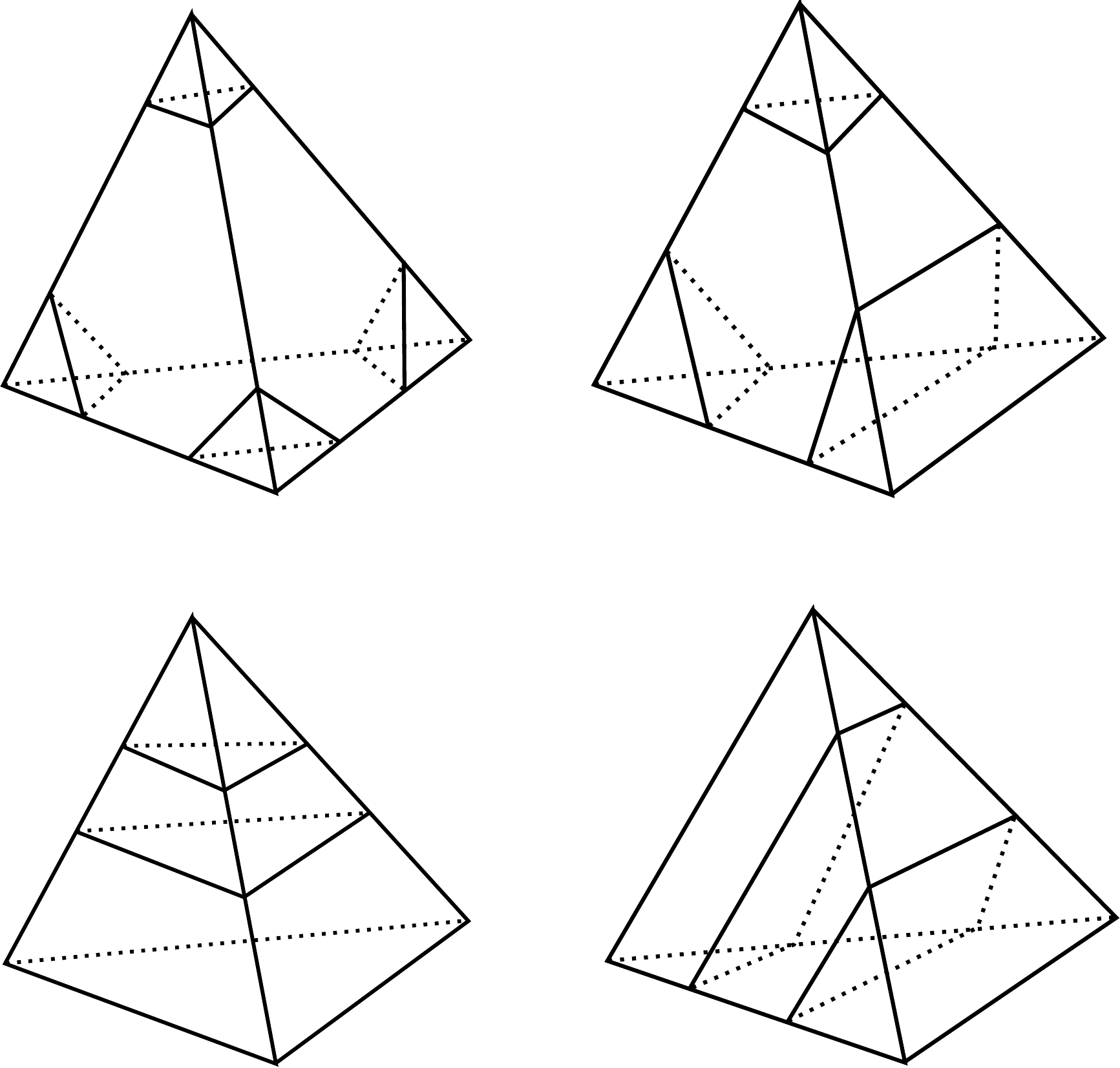}
 \caption{}
  \label{Figure9}
  \end{center}
\end{figure}

\end{proof}

We can now complete the proof of Theorem \ref{almostnormal}.  We will prove the theorem by describing a recursive process that will end when it produces an almost normal surface isotopic to the bridge surface $S_K$.  Recall that we began with a knot $K$ in a closed 3-manifold $M$ with the assumptions that $M$ and $M_K$ are irreducible.  We foliated $M_K$ by copies of the bridge surface $S_K$ with two singular leaves and triangulated $M_K$ so that the annuli $\Gamma$ are normal and the vertices of $\mathcal{T}$ are to one side of $\Gamma$.  Cutting along a maximal family of non-parallel normal 2-spheres tubed to the normal annuli $\Gamma$ we obtained the submanifold $M_0$ of $M_K$.  The surface $S_K$ induces a splitting of $M_0$ into $K_0$-compression bodies $W_0$ and $W_0'$, and $M_0$ is foliated by copies of the bridge surface $S_K$, and where the top (resp. bottom) leaf of the foliation is given by the union of the spine $\Sigma_{0}^{top}$ (resp. $\Sigma_{0}^{bot}$)  of $W_0$ (resp. $W'_0$) and $\Gamma_{0}^{top}$ (resp. $\Gamma_{0}^{bot}$).  Here $\Gamma_{0}^{top}=\Gamma^{top'}$ (resp. $\Gamma_{0}^{bot}=\Gamma^{bot'}$) are the normal annuli in $\partial M_0$.  The triple $(M_0,\Sigma_0,\Gamma_0)$ is the beginning of the recursive process.  Each later step will produce a triple $(M_i,\Sigma_i,\Gamma_i)$ such that $M_i \subset M_{i-1}$ and for each $i$ the surface $S_K$ is a weakly incompressible splitting surface for $M_i$ separating it into two $K_i$-compression bodies $W_i$ and $W'_i$, where $K_i=K\cap M_i$.  The spine $\Sigma_{i}^{top}$ (resp. $\Sigma_{i}^{bot}$)  of $W_i$ (resp. $W'_i$) is contained in some spine for $W$ (resp. $W'$), and $\Gamma_i=\partial M_i-\partial M$ is a pair of collections of normal surfaces $\Gamma_{i}^{top}$ and $\Gamma_{i}^{bot}$ in $M_i$.  

The top (bottom) leaf of a singular foliation $F_i$ is given by the intersection $\Sigma_{i}^{top} \subset \Sigma^{top}$ with $M_i$ (resp. $\Sigma_{i}^{bot} \subset \Sigma^{bot}$ with $M_i$).  It is a 1-complex in $M_i$ properly embedded in $\Gamma_{i}^{top}$. Put $\mathcal{T}^{1}_i$, the part of $\mathcal{T}^1$ lying in $M_{i}-\partial M$, in thin position with respect to $F_i$.  As mentioned earlier, either the arcs of $\mathcal{T}^{1}_i$ all have one endpoint on $\Gamma_{i}^{top}$ and one endpoint on $\Gamma_{i}^{bot}$; or there is some arc that either has both endpoints on $\Gamma_{i}^{top}$ or both endpoints on $\Gamma_{i}^{bot}$.  If there is a thick region of $\mathcal{T}^{1}_i$ in $M_i$ then we are in a position to apply Lemma \ref{Lemma5}.  Otherwise we are in a position to apply either Lemma \ref{Lemma4} or \ref{Lemma3}.  

We will describe the step that takes us from $(M_i,\Sigma_i, \Gamma_i)$ to $(M_{i+1}, \\ \Sigma_{i+1},\Gamma_{i+1})$.  First consider the initial step.  If at the first step we encounter a thick region in $M_0$, then start with a leaf $L_0$ in a thick region of $F_0$ intersecting $\mathcal{T}^{2}$ in normal arcs and simple closed curves as is guaranteed by Claim \ref{claim4}.  Applying Lemma \ref{Lemma5} we obtain a collection $G_0$ of normal surfaces and at most one almost normal surface obtained by compressing $L_0$ to one side.  If $G_0$ contains an almost normal surface and $L_0$ is incompressible above and below then $G_0=L_0$ is an almost normal surface isotopic to a leaf and we are done.  If $G_0$ does not contain an almost normal surface then without loss of generality let $G_0=\Gamma_{1}^{top}$ and proceed as below.  

Henceforth assume without loss of generality that $L_0$ compresses above to give $G_0$.  Otherwise we can invert the picture and declare $\Gamma_{i}^{bot}$ to be the ``top'' leaf.  Since $G_0$ has been obtained by compressing above, Lemma \ref{Lemma6} implies that $G_0$ is incompressible below.  By Lemma \ref{normal} we can isotope the almost normal surface $G_0$ to be normal.  This gives a new collection $\Gamma_{1}^{top}$ of normal surfaces isotopic to $G_0$.  Cut $M_0$ along the collection $\Gamma_{1}^{top}$ and keep the component to the incompressible side below $\Gamma_{1}^{top}$ that contains part of $\partial M_K$.  Call this submanifold $M_1$.  Observe that $\Gamma_{1}^{top} \subset \partial M_1$.  The cores of the tubes of the thick leaf that were compressed to give the almost normal surface $G_0 \simeq \Gamma_{1}^{top}$ form the required 1-complex $\Sigma_{1}^{top}$.  Let $\Sigma_{1}$ denote the pair $\Sigma_{1}^{top}$ and $\Sigma_{1}^{bot}=\Sigma_{0}^{bot}$, and let $\Gamma_1$ denote the pair $\Gamma_{1}^{top}$ and $\Gamma_{1}^{bot}=\Gamma_{0}^{bot}$.  This completes the first step.  

The remainder of the proof falls into the following three cases: \\

\textbf{Case 1)}:  $M_i$ contains a thick region of $\mathcal{T}^{1}_{i}$ with respect to $F_i$. 

In this case using Claim \ref{claim4} start with a leaf $L_i$ in a thick region of the foliation $F_i$ intersecting $\mathcal{T}^2$ in normal arcs and simple closed curves disjoint from the 1-skeleton.  Applying Lemma \ref{Lemma5} we obtain a collection $G_i$ of normal surfaces and at most one almost normal surface obtained by compressing $L_i$ to one side.  Lemma \ref{Lemma6} implies that $G_i$ is incompressible to the opposite side.  If $L_i$ is incompressible then $G_i=L_i$ and since $L_i$ is isotopic to a leaf we are done.  So suppose without loss of generality that $L_i$ is compressible above to give $G_i$.  The cores of the tubes of $L_i$ that are compressed above to give $G_i$ will make up the spine $\Sigma_{i+1}^{top}$.  Let $\Sigma_{i+1}$ denote the pair $\Sigma_{i+1}^{top}$ and $\Sigma_{i+1}^{bot}=\Sigma_{i}^{bot}$.  By Lemma \ref{normal} we can isotope the almost normal surface $G_i$ to give a new collection $\Gamma_{i+1}^{top}$ of normal surfaces.  Cut $M_i$ along the collection $\Gamma_{i+1}^{top}$ and keep the component to the incompressible side below $\Gamma_{i+1}^{top}$ that contains part of $\partial M_K$.  Call this new submanifold $M_{i+1}$.  Let $\Gamma_{i+1}$ denote the pair $\Gamma_{i+1}^{top}$, $\Gamma_{i+1}^{bot}=\Gamma_{i}^{bot}$.

If on the other hand $G_i$ is compressible below then the cores of the tubes of $L_i$ that are compressed below to give $G_i$ will make up the spine $\Sigma_{i+1}^{bot}$.  By Lemma \ref{normal} we can isotope $G_i$ to be normal and call the new collection of normal surfaces $\Gamma_{i+1}^{bot}$.  Let $\Gamma_{i+1}$ denote the pair $\Gamma_{i+1}^{top}=\Gamma_{i}^{top}$, and $\Gamma_{i+1}^{bot}$.  Cut $M_i$ along the collection of normal surfaces $\Gamma_{i+1}^{bot}$ and keep the component to the incompressible side above $\Gamma_{i+1}^{bot}$.  Call this new submanifolds $M_{i+1}$.  This completes the recursive step in this case.  \\

\textbf{Case 2)}:  $M_i$ contains no thick region of $\mathcal{T}_{i}^{1}$ and some arc of $\mathcal{T}_{i}^{1}$ either has both endpoints on $\Gamma_{i}^{top}$ or has both endpoints on $\Gamma_{i}^{bot}$. 

Without loss of generality suppose that there is an arc of $\mathcal{T}^{1}_{i}$ with both endpoints on $\Gamma_{i}^{top}$.  Applying Lemma \ref{Lemma4}, starting with a leaf $L_i$ of $F_i$ near the top singular leaf above all of the minima we obtain an almost normal surface $G_i$ in $M_i$ by compressing the leaf $L_i$ above.  If follows from Lemma \ref{Lemma6} that $G_i$ is incompressible below.  Moreover, $\chi(G_i)=\chi(\Gamma_{i}^{top})-2$.  Using Lemma \ref{Lemma6} isotope the almost normal surface $G_i$ to give a normal surface $\Gamma_{i+1}^{top}$.  Cut $M_i$ along $\Gamma_{i+1}^{top}$ and keep the component to the incompressible side below $\Gamma_{i+1}^{top}$ to obtain the submanifold $M_{i+1}$.  Denote by $\Gamma_{i+1}$ the pair $\Gamma_{i+1}^{top}$ and $\Gamma_{i+1}^{bot}=\Gamma_{i}^{bot}$.  The spine $\Sigma_{i+1}^{top}$ of $M_{i+1}$ consists of the cores of the tubes of $L_i$ that are compressed above to give $G_i$. Denote by $\Sigma_{i+1}$ the pair $\Sigma_{i+1}^{top}$ and $\Sigma_{i+1}^{bot}=\Sigma_{i}^{bot}$.  

In both Cases 1) and 2) the new surface $G_i$ isotopic to $\Gamma_{i+1}^{top}$ (resp. $\Gamma_{i+1}^{bot}$) in $M_i$ is not parallel as a normal surface to the normal surfaces $\Gamma_{i}^{top}$ (resp. $\Gamma_{i}^{bot}$).  The reason depends on whether Lemma \ref{Lemma5} or Lemma \ref{Lemma4} was applied.  If the surface, without loss of generality say $\Gamma_{i+1}^{top}$, comes from compressing a thick leaf via Lemma \ref{Lemma5} then there is a subarc of $\mathcal{T}^{1}$ lying between $\Gamma_{i}^{top}$ and $\Gamma_{i+1}^{top}$ with both ends on $\Gamma_{i+1}^{top}$.  Hence $\Gamma_{i}^{top}$ and $\Gamma_{i+1}^{top}$ are not parallel.  If $\Gamma_{i+1}^{top}$ comes via Lemma \ref{Lemma4} then $\chi(\Gamma_{i+1}^{top})=\chi(\Gamma_{i}^{top})-2$ so the surfaces are not parallel.  If $\Gamma_{i}^{top}$ and $\Gamma_{j}^{top}$ are parallel then all leaves $\Gamma_{k}^{top}$ where $i \leq k \leq j$ are parallel as well.  In particular, then $\Gamma_{i+1}^{top}$ is parallel to $\Gamma_{i}^{top}$ which cannot happen as we have just seen above.  Therefore it follows that $\Gamma_{i}^{top}$ is non-parallel to $\Gamma_{j}^{top}$ for all $i<j$. \\ 

\textbf{Case 3)}:  $M_i$ contains no thick region of $\mathcal{T}_{i}^{1}$ and each arc of $\mathcal{T}_{i}^{1}$ has one endpoint on $\Gamma_{i}^{top}$ and one endpoint on $\Gamma_{i}^{bot}$.

In this case by Lemma \ref{Lemma3} $M_i$ is a product.  Suppose $i\neq 0$.  The submanifold $\overline{M_i}=M_i \cup N(K_i)$ is a product as well, and has the surface $S$ as a Heegaard surface that gives an irreducible Heegaard splitting of $\overline{M_i}$.  By \cite{STh} it follows that  the splitting surface is isotopic to $\Gamma_{i}^{top}$ and  $\Gamma_{i}^{bot}$, one of which is in turn isotopic to the almost normal surface $G_{i-1}$ and so we are done. 

If $i=0$ then the argument above shows that the surface $S_K$ consists of a collection of annuli.  However the surface $S_K$ is a bridge surface for $K$ so it is connected.  Therefore $S_K$ consists of one annulus and $K$ must be the unknot.

It follows from a well known result of Haken that there are only a finite number of non-parallel, disjoint, normal surfaces in $M_K$. See \cite{H}.  Therefore we will only have to apply Lemmas \ref{Lemma4} and Lemma \ref{Lemma5} a finite number of times before we either reach a situation where we apply Lemma \ref{Lemma3} and obtain an almost normal surface isotopic to the bridge surface $S_K$ or we exhaust all of the non-parallel, disjoint, normal surfaces in $M_K$ and we obtain an almost normal surface isotopic to $S_K$.  

\end{proof}

\bibliographystyle{abbrv}

 \end{document}